\documentclass{amsart}%
\usepackage{amssymb}
\usepackage{amsfonts}
\usepackage{amsmath}
\usepackage{graphicx}%
\providecommand{\U}[1]{\protect\rule{.1in}{.1in}}
\newtheorem{theorem}{Theorem}
\theoremstyle{plain}

\newtheorem{corollary}{Corollary}

\newtheorem{definition}{Definition}
\newtheorem{example}{Example}

\newtheorem{lemma}{Lemma}

\newtheorem{proposition}{Proposition}

\numberwithin{equation}{section}
\begin{document}
\title{S-PRIME AND S-WEAKLY PRIME SUBMODULES}
\author{Emel ASLANKARAYIGIT UGURLU}
\address{Department of Mathematics, Marmara University, Kadikoy, Istanbul, 34722, Turkey}
\email{emel.aslankarayigit@marmara.edu.tr}
\date{2 September, 2019}

\begin{abstract}
In this study, all rings are commutative with non-zero identity and all
modules are considered to be unital. Let $M$ be a left $R$-module. A proper
submodule $N$ of $M$ is called an $S$-$weakly$ $prime$ submodule if $0_{M}\neq
f(m)\in N$ implies that either $m\in N$ or $f(M)\subseteq N,$ where $f\in
S=End(M)$ and $m\in M.$ Some results concerning $S$-prime and $S$-weakly prime
submodules are obtained. Then we study $S$-prime and $S$-weakly prime
submodules of multiplication modules. Also for $R$-modules $M_{1}$ and
$M_{2},$ we examine $S$-prime and $S$-weakly prime submodules of
$M=M_{1}\times M_{2},$ where $S=S_{1}\times S_{2},$ $S_{1}=End(M_{1})$ and
$S_{2}=End(M_{2}).$

\end{abstract}
\maketitle

\section{Introduction}

Throughout this paper $R$ will denote a commutative ring with a non-zero
identity and $M$ is considered to be unital left $R$-module. A proper ideal
$P$ of $R$ is said to be \textit{prime} if $ab\in P$ implies $a\in P$ or $b\in
P$, \cite{atiyah}. Weakly prime ideals in a commutative ring with non-zero
identity have been introduced and studied by D. D. Anderson and E. Smith in
\cite{DE2003}. A proper ideal $P$ of $R$ is said to be \textit{weakly prime}
if $0_{R}\neq ab\in P$ implies $a\in P$ or $b\in P$. Several authors have
extended the notion of prime ideals to modules, see, for example,
\cite{JD1978,L,MM}. In \cite{RA2003}, a proper submodule $N$ of a module $M$
over a commutative ring $R$ is said to be \textit{prime} submodule if whenever
$rm\in N$ for some $r\in R,m\in M$, then $m\in N$ or $rM\subseteq N.$ Then in
\cite{SF2007}, S. Ebrahimi and F. Farzalipour introduced weakly prime
submodules over a commutative ring $R$ as following: A proper submodule $N$ of
$M$ is called \textit{weakly prime} if for $r\in R$ and $m\in M$ with
$0_{M}\neq rm\in N$, then $m\in N$ or $rM\subseteq N.$ Clearly, every prime
submodule of a module is a weakly prime submodule. However, since $0_{M}$ is
always weakly prime, a weakly prime submodule need not be prime. Various
properties of weakly prime submodules are considered in \cite{SF2007}.

Now we define the concepts the residue of $N$ by $M$. If $N$ is a submodule of
an $R$-module $M$, the ideal $\{r\in R:rM\subseteq N\}$ is called the residue
of $N$ by $M$ and it is denoted by $(N:_{R}M).$ In particular, $(0_{M}:_{R}M)$
is called the \textit{annihilator }of $M$ and denoted by $Ann(M),$ see
\cite{ZP1988}. If the annihilator of $M$ equals to $0_{R},$ then $M$ is called
a \textit{faithful module}. Also, for a proper submodule $N$ of $M$, the
\textit{radical} of $N,$ denoted by $\sqrt{N},$ is defined to be the
intersection of all prime submodules of $M$ containing $N.$ If there is no
prime submodule containing $N$, then $\sqrt{N}=M,$ see \cite{ZP1988}.

An $R$-module $M$ is called a \textit{multiplication module} if every
submodule $N$ of $M$ has the form $IM$ for some ideal $I$ of $R$, see
\cite{ZP1988}. Note that, since $I\subseteq(N:_{R}M)$ then $N=IM\subseteq
(N:_{R}M)M\subseteq N$. So, if $M$ is multiplication, $N=(N:_{R}M)M$, for
every submodule $N$ of $M.$ Let $N$ and $K$ be submodules of a multiplication
$R$-module $M$ with $N=I_{1}M$ and $K=I_{2}M$ for some ideals $I_{1}$ and
$I_{2}$ of $R$. The product of $N$ and $K$ denoted by $NK$ is defined by
$NK=I_{1}I_{2}M$. Then by \cite[Theorem 3.4]{RA2003}, the product of $N$ and
$K$ is independent of presentations of $N$ and $K$. Note that, for
$m,m^{\prime}\in M$, by $mm^{\prime}$, we mean the product of $Rm$ and
$Rm^{\prime}$. Clearly, $NK$ is a submodule of $M$ and $NK\subseteq N\cap K,$
see \cite{RA2003}. Also, if $M$ is multiplication\ module, in Theorem 3.13 of
\cite{RA2003}, R. Ameri showed that $\sqrt{N}=\{m\in M:m^{k}\subseteq N$ for
some positive integer $k\}.$

This paper is inspired by the notion of $S$-prime submodule which appears in
\cite{GG2000, SOD2010}. The authors defined the concept as following: A proper
submodule $N$ of an $R$-module $M$ is said to be $S$-$prime$ submodule of $M$
if $f(m)\in N$ implies that either $m\in N$ or $f(M)\subseteq N,$ where $f\in
S=End(M)$ and $m\in M.$ Every $S$-$prime$ submodule is prime, see
\cite{GG2000}. For more information one can examine \cite{SOD2010}.

In this study we introduce the concept of \textit{S-weakly prime} submodule as
following: A proper submodule $N$ of an $R$-module $M$ is said to be
$S$-$weakly$ $prime$ submodule of $M$ if $0_{M}\neq f(m)\in N$ implies that
either $m\in N$ or $f(M)\subseteq N,$ where $f\in S=End(M)$ and $m\in M.$
Clearly, every $S$-$prime$ submodule is an $S$-$weakly$ $prime$ submodule. In
Proposition \ref{pro1}, it is obtained that every $S$-$weakly$ $prime$
submodule of an $R$-module $M$ is a $weakly$ $prime$ submodule. However,
we\ show that the opposite of Proposition \ref{pro1} is not correct, see
Example \ref{exa1}. Then we prove in Proposition \ref{pro4} (Proposition
\ref{pro5}) that $N$ is an $S$-prime ($S$-weakly prime) submodule if and only
if $f(K)\subseteq N$ $(0_{M}\neq f(K)\subseteq N)$ implies that either
$K\subseteq N$ or $f(M)\subseteq N,$ where $f\in S=End(M)$ and $K$ is a
submodule of $M.$ Also, we give characterizations of $S$-$prime$ submodule and
$S$-$weakly$ $prime$ submodule (Theorem \ref{theS},\ Theorem \ref{the1},
recpectively). In Corollary \ref{COR1}, by the help of Proposition \ref{pro3},
it is proved that when $N$ is an $S$-weakly prime submodule, then $(N:_{R}M)$
is an $S$-weakly prime ideal of $R.$ In multiplication module, we obtain
another characterizations for $S$-$prime$ submodule and $S$-$weakly$ $prime$
submodule (Theorem \ref{the mult1},\ Theorem \ref{the mult2}, recpectively).
Among the other results, some properties of $S$-$prime$ and $S$-$weakly$
$prime$ submodules in multiplication modules are obtained. Moreover, we
characterize $S$-$prime$ and $S$-$weakly$ $prime$ submodules of $M=M_{1}\times
M_{2}$ over $R$-module, where $M_{1},M_{2}$ be $R$-modules, see Theorem
\ref{prokart}, Theorem \ref{thekart}, Proposition \ref{pro1in2}, Proposition
\ref{pro2in2}. Finally, we obtain a relation between $S$-$prime$ and
$S$-$weakly$ $prime$ submodules of $M=M_{1}\times M_{2}$ over $R=R_{1}\times
R_{2}$-module, where $M_{1},M_{2}$ are $R_{1}$-module and $R_{2}$-module,
recpectively, see, Theorem \ref{theD1} and Theorem \ref{theD2}.

\section{S-PRIME AND S-WEAKLY PRIME SUBMODULES}

Throughout this study $End(M)$ is denoted by $S.$

\begin{definition}
A proper submodule $N$ of an $R$-module $M$ is said to be $S$-$weakly$ $prime
$ submodule of $M$ if $0_{M}\neq f(m)\in N$ implies that either $m\in N $ or
$f(M)\subseteq N,$ where $f\in S=End(M)$ and $m\in M.$
\end{definition}

It is clear that every $S$-$prime$ submodule is an $S$-$weakly$ $prime$
submodule. However, since $\{0_{M}\}$ is an $S$-$weakly$ $prime$ submodule of
$M,$ then an $S$-$weakly$ $prime$ submodule may not be an $S$-$prime$ submodule.

\bigskip Note that if we consider any $R$ as an $R$-module, then a proper
ideal $I$ of $R$ is said to be $S$-$prime$ $(S$-$weakly$ $prime)$ ideal if
$f(a)\in I$ $(0_{R}\neq f(a)\in I$) implies that either $a\in I$ or
$f(R)\subseteq I,$ where $f\in S=End(R)$ and $a\in R.$

\begin{proposition}
\label{pro1}Every $S$-$weakly$ $prime$ submodule of an $R$-module $M$ is a
$weakly$ $prime$ submodule.
\end{proposition}

\begin{proof}
Let $N$ be an $S$-$weakly$ $prime$ submodule of an $R$-module $M.$ Suppose
that for some $r\in R$ and $m\in M$ such that $0_{M}\neq rm\in N$ and $m\notin
N.$ We show that $r\in(N:_{R}M).$

Define $h:M\rightarrow M$ such that $h(x)=rx$ for all $x\in M.$ Clearly, $h\in
End(M).$ Since definition of $h$ and our assumption, then $0_{M}\neq
h(m)=rm\in N.$ Then as $m\notin N$ and $N$ is an $S$-$weakly$ $prime$
submodule, we get $h(M)\subseteq N.$ Thus $h(M)=rM\subseteq N,$ i.e.,
$r\in(N:_{R}M).$
\end{proof}

Note that\ generally a $weakly$ $prime$ submodule is not an $S$-$weakly$
$prime$ submodule. For this one can see the following example:

\begin{example}
\label{exa1}Let consider the submodule $N=2%
\mathbb{Z}
\oplus%
\mathbb{Z}
$ of $%
\mathbb{Z}
$-module $M=%
\mathbb{Z}
\oplus%
\mathbb{Z}
.$ Since $N$ is a maximal submodule, $N$ is a prime, so weakly prime
submodule. But $N$ is not an $S$-$weakly$ $prime$ submodule. Indeed, let
define $f:M\longrightarrow M$ such that $f((x,y))=(y,x)$ for all $(x,y)\in M.$
Then we get $f\in S=End(M).$ Thus we can easily see $0_{M}\neq
f((1,2))=(2,1)\in N,$ but $(1,2)\notin N$ and $f(M)=M\nsubseteq N.$
Consequently, $N$ is not $S$-$weakly$ $prime.$
\end{example}

\begin{proposition}
\label{pro4}Let $M$ be an $R$-module and $N$ be a proper submodule of $M.$
Then the followings are equivalent:

\begin{enumerate}
\item $N$ is an $S$-prime submodule.

\item $f(K)\subseteq N$ implies that either $K\subseteq N$ or $f(M)\subseteq
N,$ where $f\in S$ and $K$ is a proper submodule of $M.$
\end{enumerate}
\end{proposition}

\begin{proof}
$(1)\Longrightarrow(2):$ Let $N$ be an $S$-prime submodule. Assume that
$f(K)\subseteq N$ and $K\nsubseteq N.$ Then there exists $k\in K-N.$ Thus
$f(k)\in f(K)\subseteq N.$ Since $N$ is $S$-prime, $f(M)\subseteq N.$

$(2)\Longrightarrow(1):$ Let $f(m)\in N.$ We show that either $m\in N$ or
$f(M)\subseteq N.$ Since $f(Rm)\subseteq N,$ by our hypothesis we obtain
either $Rm\subseteq N$ or $f(M)\subseteq N.$ Consequently, $m\in N$ or
$f(M)\subseteq N.$
\end{proof}

\begin{corollary}
Let $M$ be an $R$-module and $N$ be a proper submodule of $M.$ Then the
followings are equivalent:
\end{corollary}

\begin{enumerate}
\item $N$ is an $S$-prime submodule.

\item $f(Rm)\subseteq N$ implies that either $Rm\subseteq N$ or $f(M)\subseteq
N,$ where $f\in S$ and $m\in M.$
\end{enumerate}

\begin{proof}
By Proposition \ref{pro4}.
\end{proof}

\begin{proposition}
\label{pro babei}Let $M$ be an $R$-module and $N$ be a proper submodule of
$M.$ Then $N$ is an $S$-prime submodule if and only if \ $f^{-1}(N)=M$ or
$f^{-1}(N)\subseteq N,$ for all $f\in S.$
\end{proposition}

\begin{proof}
$\Longrightarrow:$ Let $N$ be an $S$-prime submodule. Assume that
$f(M)\subseteq N.$ Then it is clear that \ $f^{-1}(N)=M.$ So suppose that
$f(M)\nsubseteq N.$ Take $m\in f^{-1}(N).$ Then $f(m)\in N.$ Since $N$ is an
$S$-prime submodule and $f(M)\nsubseteq N,$ we have $m\in N.$ Consequently,
$f^{-1}(N)\subseteq N.$

$\Longleftarrow:$ Assume that $f^{-1}(N)\subseteq N$ or $f^{-1}(N)=M$ for all
$f\in End(M).$ Let $f(m)\in N.$ If $f^{-1}(N)\subseteq N,$ then $m\in
f^{-1}(N)\subseteq N,$ so we are done. On the other hand, if $f^{-1}(N)=M$,
then we get $f(M)\subseteq N.$ Thus $N$ is an $S$-prime submodule.
\end{proof}

\begin{corollary}
The zero submodule $\{0_{M}\}$ of $M$ is an $S$-prime submodule if and only
if\ $f$ is one-to-one, for all $0\neq f\in S.$
\end{corollary}

\begin{proof}
By Proposition \ref{pro babei}.
\end{proof}

\begin{theorem}
\label{the1}Let $M$ be an $R$-module and $N$ be a proper submodule of $M.$ For
all $f\in S,$ the followings are equivalent:

\begin{enumerate}
\item $N$ is an $S$-$weakly$ $prime$ submodule of $M.$

\item $(N:_{R}f(x))=(N:_{R}f(M))\cup(0_{M}:_{R}f(x))$ for all $x\notin N.$

\item $(N:_{R}f(x))=(N:_{R}f(M))$ or $(N:_{R}f(x))=(0_{M}:_{R}f(x))$ for all
$x\notin N.$
\end{enumerate}
\end{theorem}

\begin{proof}
$(1)\Longrightarrow(2):$ Assume that $N$ is $S$-$weakly$ $prime$. Let
$r\in(N:_{R}f(x))$ and $x\notin N.$ Then we get $rf(x)\in N.$ If
$rf(x)=0_{M},$ then $r\in(0_{M}:_{R}f(x)).$ Suppose that $rf(x)\neq0_{M}.$
Define $h:M\rightarrow M$ such that $h(m)=rf(m),$ for all $m\in M.$ Clearly
$h\in End(M)$, also $0_{M}\neq h(x)=rf(x)\in N.$ Since $N$ is an $S$-$weakly$
$prime$ submodule and $x\notin N,$ then we obtain $h(M)\subseteq N.$ Thus
$h(M)=rf(M)\subseteq N$ and so $r\in(N:_{R}f(M)).$

$(2)\Longrightarrow(3):$ Clear.

$(3)\Longrightarrow(1):$ Suppose that $h\in End(M)$ and $m\notin N$ such that
$0_{M}\neq h(m)\in N.$ We prove that $h(M)\subseteq N.$ Since $0_{M}\neq
h(m),$ we get $(N:_{R}h(m))\neq(0_{M}:_{R}h(m)).$ Indeed, if $(N:_{R}%
h(m))=(0_{M}:_{R}h(m)),$ then we obtain $(N:_{R}h(m))=R=(0_{M}:_{R}h(m)),$
i.e., $1_{R}h(m)=0_{M},$ a contradiction. Thus we have $(N:_{R}h(m))=(N:_{R}%
h(M)),$ by our hypothesis $(3)$. Since $(N:_{R}h(m))=R,$ we get $h(M)\subseteq
N.$
\end{proof}

\begin{proposition}
\label{pro5}Let $M$ be an $R$-module and $N$ be a proper submodule of $M$.
Then the followings are equivalent:

\begin{enumerate}
\item $N$ is an $S$-weakly prime submodule.

\item $0_{M}\neq f(K)\subseteq N$ implies that either $K\subseteq N$ or
$f(M)\subseteq N,$ where $f\in S$ and $K$ is a submodule of $M.$
\end{enumerate}
\end{proposition}

\begin{proof}
$(1)\Longrightarrow(2):$ Let $N$ be an $S$-weakly prime submodule. Suppose
that $f(K)\subseteq N$, $K\nsubseteq N$ and $f(M)\nsubseteq N.$ Then we show
$f(K)=0_{M}.$ For every $k\in K,$ we have 2 cases:

Case 1: Let $k\in K-N.$ By Theorem \ref{the1}, we can see that $(N:_{R}%
f(k))=(N:_{R}f(M))$ or $(N:_{R}f(k))=(0_{M}:_{R}f(k)).$ Since $f(k)\in
f(K)\subseteq N,$ one get $1_{R}\in(N:_{R}f(k)).$ Thus either $1_{R}\in
(N:_{R}f(M))$ or $1_{R}\in(0_{M}:_{R}f(k)).$ The first one contradicts our
assumption. Thus we obtain $f(k)=0_{M}.$

Case 2: Let $k\in K\cap N.$ If $f(k)=0_{M},$ we are done. Let suppose
$f(k)\neq0_{M}.$ Since $K\nsubseteq N$, there exists $0_{M}\neq y\in K-N$.
\ Then $f(y)\in f(K)\subseteq N,$ one get $1_{R}\in(N:_{R}f(y)).$ Thus either
$1_{R}\in(N:_{R}f(M))$ or $1_{R}\in(0_{M}:_{R}f(y)).$ The first one
contradicts our assumption. Thus we obtain $f(y)=0_{M}.$ Then one can see
$0_{M}\neq f(y+k)=f(y)+f(k)\in f(K)\subseteq N.$ Since $N$ is $S$-weakly
prime, we get $y+k\in N$ or $f(M)\subseteq N.$ So, $y\in N$ or $f(M)\subseteq
N,$ i.e., contradiction.

Consequently, for every $k\in K,$ we obtain $f(k)=0_{M}.$

$(2)\Longrightarrow(1):$ Let $0_{M}\neq f(m)\in N.$ We show that either $m\in
N$ or $f(M)\subseteq N.$ Since $0_{M}\neq f(Rm)\subseteq N,$ by our hypothesis
we obtain either $Rm\subseteq N$ or $f(M)\subseteq N.$ Consequently, $m\in N$
or $f(M)\subseteq N.$
\end{proof}

\begin{corollary}
Let $M$ be an $R$-module and $N$ be a proper submodule of $M$. Then the
followings are equivalent:

\begin{enumerate}
\item $N$ is an $S$-weakly prime submodule.

\item $0_{M}\neq f(Rm)\subseteq N$ implies that either $Rm\subseteq N$ or
$f(M)\subseteq N,$ where $f\in S$ and $m\in M.$
\end{enumerate}
\end{corollary}

\begin{proof}
By Proposition \ref{pro5}.
\end{proof}

\begin{theorem}
\label{theS}Let $M$ be an $R$-module and $N$ be a proper submodule of $M.$ For
all $f\in S,$ the followings are equivalent:

\begin{enumerate}
\item $N$ is an $S$-$prime$ submodule of $M.$

\item $(N:_{R}f(x))=(N:_{R}f(M))$ for all $x\notin N.$
\end{enumerate}
\end{theorem}

\begin{proof}
$(1)\Longrightarrow(2):$ Assume that $N$ is $S$-$prime$ and $x\notin N$. Let
$r\in(N:_{R}f(x)).$ Then we get $rf(x)\in N.$ Define $h:M\rightarrow M$ such
that $h(m)=rf(m)$ for all $m\in M.$ Clearly $h\in End(M)$, also $h(x)=rf(x)\in
N.$ Since $N$ is an $S$-$prime$ submodule and $x\notin N,$ then we obtain
$h(M)\subseteq N.$ Thus $h(M)=rf(M)\subseteq N$ and so $r\in(N:_{R}f(M)).$ The
other containment is always hold.

$(2)\Longrightarrow(1):$ Suppose that $h\in End(M)$ and $m\notin N$ such that
$h(m)\in N.$ We prove that $h(M)\subseteq N.$ Since $1_{R}h(m)\in N,$ we get
$1_{R}\in(N:_{R}h(m))=$ $(N:_{R}h(M)).$ Thus $h(M)\subseteq N$ by the assumption.
\end{proof}

To avoid losing the integrity, we give the following proposition.

\begin{proposition}
\label{pro2}(\cite{SF2007}, Proposition 2.1) Let $M$ be a faithful cyclic
$R$-module. If $N$ is a weakly prime submodule, then $(N:_{R}M)$ is a weakly
prime ideal of $R.$
\end{proposition}

However, Proposition \ref{pro2} is not true for "$S$-weakly prime situation".
So we mean if $N$ is an $S$-weakly prime submodule, then $(N:_{R}M)$ may not
be an $S$-weakly prime ideal of $R.$

Note that for a subset $A$ of $M$, we denote the submodule generated by $A$ in
$M$ as $<A>$. In particular, if $X=\{a\},$ then it is denoted by $<a>$. If $M$
is an $R$-module such that $M=<a>,$ then $M$ is called cyclic module. It is
clear that every cyclic module is a multiplication module, see \cite{PF1988}.

\begin{proposition}
\label{pro3}Let $M$ be a cyclic $R$-module such that $M=<a>$ and $N$ be a
proper submodule of $M.$ Then the followings are equivalent:

\begin{enumerate}
\item $N$ is a weakly prime submodule.

\item $N$ is an $S$-weakly prime submodule.
\end{enumerate}
\end{proposition}

\begin{proof}
$(1)\Longrightarrow(2):$ Assume that $N$ is a weakly prime submodule. Let
choose $m\in M$ and $f\in End(M)$ such that $0_{M}\neq f(m)\in N$ and $m\notin
N.$ We prove that $f(M)\subseteq N.$ Let $f(x)\in f(M).$ Since $M=<a>,$ there
exist $r_{1},r_{2}\in R\ $such that $x=r_{1}a$ and $m=r_{2}a.$ Then we get
$0_{M}\neq f(m)=f(r_{2}a)=r_{2}f(a)\in N.$ Since $N$ is weakly prime, then
$r_{2}\in(N:_{R}M)$ or $f(a)\in N.$ If $r_{2}\in(N:_{R}M),$ then we get
$m=r_{2}a\in N,$ i.e., a contradiction. Thus $f(a)\in N,$ so $f(x)=r_{1}%
f(a)\in N.$ As $x$ is an arbitrary element of $M$, we obtain $f(M)\subseteq
N.$

$(2)\Longrightarrow(1):$ By Proposition \ref{pro1}.
\end{proof}

\begin{corollary}
\label{COR1}Let $M$ be a faithful cyclic $R$-module. If $N$ is an $S$-weakly
prime submodule, then $(N:_{R}M)$ is an $S$-weakly prime ideal of $R.$
\end{corollary}

\begin{proof}
Assume that $N$ is an $S$-weakly prime submodule. Then $N$ is an weakly prime
submodule. By Proposition \ref{pro2}, $(N:_{R}M)$ is a weakly prime ideal of
$R.$ Since $R$ is a cyclic $R$-module, $(N:_{R}M)$ is an $S$-weakly prime
ideal of $R$ by Proposition \ref{pro3}.
\end{proof}

\bigskip

For the integrity of our study, we give the following Lemma:

\begin{lemma}
\label{lemma fg}(\cite{PF1988},Corollary in page 231) Let $I$, $J$ be two
ideals of $R$ and $M$ be a finitely generated multiplication $R$-module. Then
$IM\subseteq JM$ if and only if $I\subseteq J+Ann(M).$
\end{lemma}

\begin{theorem}
\label{the mult1}Let $M$ be a finitely generated multiplication $R$-module and
$N$ be a proper submodule of $M.$ Then the followings are equivalent:

\begin{enumerate}
\item $N$ is an $S$-prime submodule.

\item $(N:_{R}M)$ is an $S$-prime ideal of $R.$

\item $N=IM,$ for some $S$-prime ideal $I$ of $R$ with $Ann(M)\subseteq I.$
\end{enumerate}
\end{theorem}

\begin{proof}
$(1)\Longrightarrow(2):$ By Corollary 2.1.5 in \cite{SOD2010}.

$(2)\Longrightarrow(3):$ Since $M$ is a multiplication module, $N=(N:_{R}M)M,$
so we are done.

$(3)\Longrightarrow(1):$ Suppose that $N=IM,$ for some $S$-prime ideal $I$ of
$R$ with $Ann(M)\subseteq I.$ To use Proposition \ref{pro4}, assume that $K$
is a submodule of $M$ such that $f(K)\subseteq N,$ for any $f\in S.$ Since $M$
is a multiplication module, there exist two ideals $J_{1},J_{2}$ of $R$ such
that $K=J_{1}M$ and $f(M)=J_{2}M.$ Then $f(J_{1}M)=J_{1}f(M)=J_{1}%
J_{2}M\subseteq N=IM.$ By Lemma \ref{lemma fg}, $J_{1}J_{2}\subseteq
I+Ann(M)=I.$ As $I$ is an $S$-prime ideal, so prime, we get $J_{1}\subseteq I$
or $J_{2}\subseteq I.$ It implies $J_{1}M\subseteq N$ or $J_{2}M\subseteq N.$
Consequently, $K\subseteq N$ or $f(M)\subseteq N.$
\end{proof}

\begin{theorem}
\label{the mult2}Let $M$ be a cyclic faithful $R$-module and $N$ be a proper
submodule of $M$. Then the followings are equivalent:

\begin{enumerate}
\item $N$ is an $S$-weakly prime submodule.

\item $(N:_{R}M)$ is an $S$-weakly prime ideal of $R.$

\item $N=IM,$ for some $S$-weakly prime ideal $I$ of $R.$
\end{enumerate}
\end{theorem}

\begin{proof}
$(1)\Longrightarrow(2):$ By Corollary \ref{COR1}.

$(2)\Longrightarrow(3):$ Since $M$ is a multiplication module, $N=(N:_{R}M)M,$
so we are done.

$(3)\Longrightarrow(1):$ \ By the help of Proposition \ref{pro5} and Lemma
\ref{lemma fg}, as the previous proof one can prove easily.
\end{proof}

\begin{definition}
(\cite{SOD2010}, Definition 2.1.1)A proper submodule $N$ of an $R$-module $M$
is said to be fully invariant submodule of $M$ if $f(N)\subseteq N,$ for every
$f\in S$.
\end{definition}

\begin{theorem}
\label{the3}Let $M$ be an $R$-module and $N$ be a fully invariant and
$S$-weakly prime submodule of $M$ that is not $S$-prime. If $I$ is an ideal of
$R$ such that $I\subseteq(N:_{R}M),$ then $If(N)=0_{M},$ for any $f\in S.$ In
particular, $(N:_{R}M)f(N)=0_{M}.$
\end{theorem}

\begin{proof}
Suppose that $If(N)\neq0_{M}.$ We show that $N$ is an $S$-prime submodule. Let
$f(m)\in N,$ where $f\in End(M)$ and $m\in M.$ If \ $f(m)\neq0_{M},$ since $N$
is $S$-weakly prime, we are done. So, assume that $f(m)=0_{M}.$ Then we have 2
cases for $f(N).$

Case 1: $f(N)\neq0_{M}.$ Then there exists $n\in N$ such that $f(n)\neq0_{M}.$
Thus $0_{M}\neq f(n+m).$ As $N$ is fully invariant, one see $0_{M}\neq
f(n+m)\in N.$ Since $N$ is $S$-weakly prime, $m+n\in N,$ i.e., $m\in N$ or
$f(M)\subseteq N.$

Case 2: $f(N)=0_{M}.$ As $If(N)\neq0_{M},$ contradiction.
\end{proof}

\begin{corollary}
\label{cor n2}Let $M$ be an $R$-module and $N$ be a fully invariant and
$S$-weakly prime submodule of $M$ that is not $S$-prime. If $M$ is
multiplication, then $f(N)^{2}=0_{M},$ for any $f\in S.$
\end{corollary}

\begin{proof}
Let $M$ be multiplication. Then $f(N)^{2}=(f(N):_{R}M)M(f(N):_{R}%
M)M=(f(N):_{R}M)(f(N):_{R}M)M=(f(N):_{R}M)f(N)\subseteq(N:_{R}M)f(N)$, since
$N$ is fully invariant. By Theorem \ref{the3}, $(N:_{R}M)f(N)=0_{M},$ so
$f(N)^{2}=0_{M}.$
\end{proof}

\begin{corollary}
Let $M$ be a multiplication $R$-module and $N,K$ be fully invariant and
$S$-weakly prime submodules of $M$ that are not $S$-prime. Then
$f(N)f(K)\subseteq\sqrt{0_{M}}.$
\end{corollary}

\begin{proof}
Assume that $f(b)\in f(K).$ Then $f(b)^{2}=Rf(b)Rf(b)\subseteq f(K)^{2}%
=0_{M},$ by Corollary \ref{cor n2}. Then we get $f(K)\subseteq\sqrt{0_{M}}.$
Similarly, $f(N)\subseteq\sqrt{0_{M}}.$ Then we obtain $f(N)f(K)\subseteq
\sqrt{0_{M}}.$
\end{proof}

For the next proof, we will need the following Lemma:

\begin{lemma}
\label{lemma atani}(\cite{SF2007}, Lemma 2.5) Let $M$ be a multiplication
module over $R$. Let $N$ and $K$ be submodules of $M$. Then the followings are hold:

\begin{enumerate}
\item If for every $a\in N,aK=0_{M},$ then $NK=0_{M}.$

\item If for every $b\in K,Nb=0_{M},$ then $NK=0_{M}.$

\item If for every $a\in N,b\in K,ab=0_{M},$ then $NK=0_{M}.$
\end{enumerate}
\end{lemma}

\begin{theorem}
\label{the4}Let $M$ be a finitely generated faithful multiplication $R$-module
and $N$ be a fully invariant and $S$-weakly prime submodule of $M$ that is not
$S$-prime. If $f\in S$ is onto, then $f(N)f(\sqrt{0_{M}})=0_{M}.$
\end{theorem}

\begin{proof}
Let $y=f(x)\in f(\sqrt{0_{M}})$ such that $x\in\sqrt{0_{M}}.$ Then there
exists two ideals $I,J$ in $R$ such that $f(N)=IM$ and $Rx=JM$. Then as $f$ is
onto, one see that $Rf(x)=f(Rx)=f(JM)=Jf(M)=JM.$ For $f(x),$ we have 2 cases :

Case 1: Let $f(x)\in f(N).$ Then $Rf(x)\subseteq f(N).$ By Lemma
\ref{lemma fg}, we get $J\subseteq I.$ Thus with Corollary \ref{cor n2}, we
have $f(N)Rf(x)=IJM\subseteq f(N)^{2}=0_{M}.$ By Lemma \ref{lemma atani},
$f(N)f(\sqrt{0_{M}})=0_{M}.$

Case 2: Let $f(x)\notin f(N).$ Then we get $x\notin N.$ By Theorem \ref{the1},
$(N:_{R}f(x))=(0_{M}:_{R}f(x))$ or $(N:_{R}f(x))=(N:_{R}f(M)).$

Assume that $(N:_{R}f(x))=(0_{M}:_{R}f(x)).$ Thus $(N:_{R}M)M\subseteq
(N:_{R}f(x))M=(0_{M}:_{R}f(x))M,$ so, as $N$ is fully invariant,
$f(N)\subseteq(0_{M}:_{R}f(x))M$. Then $f(N)Rf(x)\subseteq(0_{M}%
:_{R}f(x))Rf(x)=0_{M},$ i.e., $f(N)f(x)=0_{M}.$ By Lemma \ref{lemma atani},
$f(N)f(\sqrt{0_{M}})=0_{M}.$

Now, suppose that $(N:_{R}f(x))=(N:_{R}f(M)).$ As $x\in\sqrt{0_{M}},$ there
exists a smallest positive integer $n$ such that $x^{n}=0_{M}$ and
$x^{n-1}\neq0_{M}.$ Then we see $Rx^{n}=J^{n}M=0_{M}$ and $Rx^{n-1}%
=J^{n-1}M\neq0_{M}.$ Hence, since $Rf(x)=JM,$ one get $Rf(x)^{n}=J^{n}M=0_{M}$
and $Rf(x)^{n-1}=J^{n-1}M\neq0_{M}.$ Moreover, we have $J^{n}M=0_{M}$ implies
$J^{n}=0_{R},$ by Lemma \ref{lemma fg}. Then it is clear that $J^{n-1}%
\subseteq(I:_{R}J).$ Hence, as $N$ is fully invariant, $0_{M}\neq
Rf(x)^{n-1}=J^{n-1}M\subseteq(IM:_{R}JM)M=(f(N):_{R}Rf(x))M\subseteq
(N:_{R}Rf(x))M.$ Then by our hypothesis and as $f$ is onto, $0_{M}\neq
Rf(x)^{n-1}\subseteq(N:_{R}Rf(x))M=(N:_{R}f(M))M=(N:_{R}M)M=N,$ this implies
$0_{M}\neq f(x^{n-1})\subseteq N.$ Since $N$ is $S$-weakly prime, we get
$0_{M}\neq x^{n-1}\subseteq N$ or $f(M)\subseteq N.$ The second one
contradicts with $f(x)\notin N.$ As every $S$-weakly prime is a weakly prime
submodule and by Theorem 2.6 in \cite{SF2007}, $0_{M}\neq Rx^{n-1}\subseteq N$
implies $Rx\subseteq N,$ so $f(x)\in f(N),$ a contradiction.
\end{proof}

\begin{corollary}
Let $M$ be a finitely generated faithful multiplication $R$-module and $N,K$
be fully invariant and $S$-weakly prime submodules of $M$ that are not prime.
If $f\in S$ is onto, then $f(N)f(K)=0_{M}.$
\end{corollary}

\begin{proof}
Assume that $f(b)\in f(K).$ Then $f(b)^{2}\subseteq f(K)^{2}=0_{M},$ by
Corollary \ref{cor n2}. Then we get $f(K)\subseteq\sqrt{0_{M}}.$ As $N$ is
fully invariant, one see $f(N)f(K)\subseteq f(N)\sqrt{0_{M}}\subseteq
N\sqrt{0_{M}}.$ Since $N$ is $S$-weakly prime (so weakly prime) and not prime,
we know $N\sqrt{0_{M}}=0_{M}$, by the help of \ Theorem 2.7 in \cite{SF2007}.
Consequently, $f(N)f(K)=0_{M}.$
\end{proof}

\begin{corollary}
\label{cor0}Let $M$ be a finitely generated faithful multiplication module
over $R$ with unique maximal submodule $K$ and every prime of $M$ is maximal.
Let $N$ be a fully invariant and $S$-weakly prime submodule of $M$. If $f\in
S$ is onto, then $N=K$ or $f(N)f(K)=0_{M}.$
\end{corollary}

\begin{proof}
If $N$ is $S$-prime, so prime, by our hypothesis $N=K.$ If $N$ is not
$S$-prime, one see $\sqrt{0_{M}}=%
{\displaystyle\bigcap\limits_{N_{i}\in Spec(M)}}
N_{i}=K$, where $Spec(M)$ is the set of all prime submodules of $M.$ Then we
obtain $f(\sqrt{0_{M}})=f(K).$ Thus $f(N)f(K)=$ $f(N)f(\sqrt{0_{M}})=0_{M},$
by Theorem \ref{the4}.
\end{proof}

\begin{corollary}
Let $M$ be a finitely generated faithful module over a local ring $R$ with
unique maximal submodule $K$ and every prime of $M$ is maximal. Let $N$ be a
fully invariant and $S$-weakly prime submodule of $M$. If $f\in S$ is onto,
then $N=K$ or $f(N)f(K)=0_{M}.$
\end{corollary}

\begin{proof}
By Corollary 1 in \cite{CPL1995}, $M$ is cyclic. Thus $M$ is multiplication
$R$-module. By Corollary \ref{cor0}, it is done.
\end{proof}

\bigskip

Let $M_{1},M_{2}$ be $R$-modules and we know that $M=M_{1}\times M_{2}$ is an
$R$-module. For every $f_{i}\in End(M_{i}),$ let define $f:M\rightarrow M$
with $f((m_{1},m_{2}))=(f_{1}(m_{1}),f_{2}(m_{2}))$, for every $(m_{1}%
,m_{2})\in M,$ $i=1,2.$ Then one can easily see, $f\in End(M)$ and
$f(M)=f_{1}(M_{1})\times f_{2}(M_{2}).$ Also, we use the following notations:
$S_{1}=End(M_{1})$, $S_{2}=End(M_{2})$ and $S=S_{1}\times S_{2}.$

\begin{theorem}
\label{prokart}Let $M_{1},M_{2}$ be $R$-modules and $N_{1}$ be a proper
submodule of $M_{1}$. Then the followings are equivalent:

\begin{enumerate}
\item $N=N_{1}\times M_{2}$ is an $S$-prime submodule of $M=M_{1}\times
M_{2}.$

\item $N_{1}$ is an $S_{1}$-prime submodule of $M_{1}.$
\end{enumerate}
\end{theorem}

\begin{proof}
$(1)\Longrightarrow(2):$ Suppose that $N=N_{1}\times M_{2}$ is an $S$-prime
submodule of $M=M_{1}\times M_{2}.$ Let $f_{1}(m_{1})\in N_{1},$ for some
$m_{1}\in M_{1}$ and $f_{1}\in End(M_{1}).$ Then for every $m_{2}\in M_{2}$
and $f_{2}\in End(M_{2}),$ we get $f((m_{1},m_{2}))=(f_{1}(m_{1}),f_{2}%
(m_{2}))\in N_{1}\times M_{2}=N.$ So, as $N=N_{1}\times M_{2}$ is an $S$-prime
submodule of $M=M_{1}\times M_{2},$ we get $(m_{1},m_{2})\in N$ or
$f(M_{1}\times M_{2})\subseteq N.$ Thus, by $f(M)=f_{1}(M_{1})\times
f_{2}(M_{2}),$ one see $m_{1}\in N_{1}$ or $f_{1}(M_{1})\subseteq N_{1},$
i.e., $N_{1}$ is $S_{1}$-prime.

$(2)\Longrightarrow(1):$ Let $N_{1}$ be an $S_{1}$-prime submodule of $M_{1}.$
Assume that $f((m_{1},m_{2}))=(f_{1}(m_{1}),f_{2}(m_{2}))\in N=N_{1}\times
M_{2},$ for some $(m_{1},m_{2})\in M,f_{1}\in End(M_{1})$ and $f_{2}\in
End(M_{2}).$ Then $f_{1}(m_{1})\in N_{1}.$ As $N_{1}$ is $S_{1}$-prime, we
have $m_{1}\in N_{1}$ or $f_{1}(M_{1})\subseteq N_{1}.$ Thus $(m_{1},m_{2})\in
N$ or $f(M_{1}\times M_{2})=f_{1}(M_{1})\times f_{2}(M_{2})\subseteq
N=N_{1}\times M_{2}.$
\end{proof}

\begin{theorem}
\label{thekart}Let $M_{1},M_{2}$ be $R$-modules and $N_{1}$ be a proper
submodule of $M_{1}$. Then the followings are equivalent:

\begin{enumerate}
\item $N=N_{1}\times M_{2}$ is an $S$-weakly prime submodule of $M=M_{1}\times
M_{2}.$

\item $N_{1}$ is $S_{1}$-weakly prime and for every $(m_{1},m_{2})\in M$,
$f_{1}\in S_{1}$ and $f_{2}\in S_{2},$ if $f_{1}(m_{1})=0_{M_{1}},m_{1}\notin
N_{1}$ and $f_{1}(M_{1})\nsubseteq N_{1},$ then $f_{2}(m_{2})=0_{M_{2}}.$
\end{enumerate}
\end{theorem}

\begin{proof}
$(1)\Longrightarrow(2):$ Assume that $N=N_{1}\times M_{1}$ is an $S$-weakly
prime submodule of $M=M_{1}\times M_{2}.$ Firstly, we show that $N_{1}$ is
$S_{1}$-weakly prime. Let $0_{M_{1}}\neq f_{1}(m_{1})\in N_{1},$ for some
$m_{1}\in M_{1}.$ Then for every $m_{2}\in M_{2},$ we get $0_{M}\neq
f((m_{1},m_{2}))=(f_{1}(m_{1}),f_{2}(m_{2}))\in N.$ So, as $N=N_{1}\times
M_{2}$ is an $S$-weakly prime submodule of $M=M_{1}\times M_{2},$ we get
$(m_{1},m_{2})\in N$ or $f(M_{1}\times M_{2})\subseteq N.$ Thus $m_{1}\in
N_{1}$ or $f_{1}(M_{1})\subseteq N_{1},$ i.e., $N_{1}$ is $S_{1}$-weakly
prime. Let $0_{M_{1}}=f_{1}(m_{1})\in N_{1}$ such that $m_{1}\notin N_{1}$ and
$f_{1}(M_{1})\nsubseteq N_{1}.$ Then for every $m_{2}\in M_{2},$ we say
$(m_{1},m_{2})\notin N$ and $f(M_{1}\times M_{2})\nsubseteq N.$ Moreover, if
$f_{2}(m_{2})\neq0_{M_{2}},$ we have $0_{M}\neq f((m_{1},m_{2}))=(f_{1}%
(m_{1}),f_{2}(m_{2}))\in N,$ so $(m_{1},m_{2})\in N$ or $f(M_{1}\times
M_{2})\subseteq N,$ a contradiction. Consequently, $f_{2}(m_{2})=0_{M_{2}}.$

$(2)\Longrightarrow(1):$ Suppose that the condition (2) is true.

Let $0_{M}\neq f((m_{1},m_{2}))=(f_{1}(m_{1}),f_{2}(m_{2}))\in N,$ for some
$(m_{1},m_{2})\in M.$ Then for $f_{1}(m_{1}),$ we have 2 cases:

Case 1: Let $0_{M_{1}}\neq f_{1}(m_{1}).$ Since $f(m_{1})\in N_{1}$ and
$N_{1}$ is $S_{1}$-weakly prime, we get $m_{1}\in N_{1}$ or $f(M_{1})\subseteq
N_{1},$ i.e., $(m_{1},m_{2})\in N$ or $f(M_{1}\times M_{2})\subseteq N.$ Thus,
it is done.

Case 2: Let $0_{M_{1}}=f_{1}(m_{1}).$ Then $0_{M_{2}}\neq f_{2}(m_{2}).$
Assume that $m_{1}\notin N_{1}$ and $f_{1}(M_{1})\nsubseteq N_{1}.$ Then by
(2), $f_{2}(m_{2})=0_{M_{2}},$ a contradiction. So $m_{1}\in N_{1}$ or
$f_{1}(M_{1})\subseteq N_{1}.$ Thus $(m_{1},m_{2})\in N$ or $f(M)\subseteq N.$
\end{proof}

\begin{proposition}
\label{pro1in2}Let $M_{1},M_{2}$ be $R$-modules and $N_{1},N_{2}$ be a proper
submodule of $M_{1},M_{2},$ resceptively. If $N=N_{1}\times N_{2}$ is an
$S$-prime submodule of $M=M_{1}\times M_{2}$, then $N_{1}$ is an $S_{1}$-prime
submodule of $M_{1}$ and $N_{2}$ is an $S_{2}$-prime submodule of $M_{2}.$
\end{proposition}

\begin{proof}
Suppose that $N=N_{1}\times N_{2}$ is an $S$-prime submodule of $M=M_{1}\times
M_{2}.$ Now, we will show that $N_{1}$ is an $S_{1}$-prime submodule. Take
$f_{1}\in S_{1}$ such that $f_{1}(m_{1})\in N_{1}$ for some $m_{1}\in M_{1}.$
Also take $f_{2}\in S_{2}.$ Then $f((m_{1},0_{M_{2}}))=(f_{1}(m_{1}%
),f_{2}(0_{M_{2}}))=(f_{1}(m_{1}),0_{M_{2}})\in N_{1}\times N_{2}=N.$ Since
$N$ is $S$-prime, $(m_{1},0_{M_{2}})\in N$ or $f(M_{1}\times M_{2})\subseteq
N.$ This implies that $m_{1}\in N_{1}$ or $f_{1}(M_{1})\subseteq N_{1}.$
Similar argument shows that $N_{2}$ is an $S_{2}$-prime submodule.
\end{proof}

\begin{proposition}
\label{pro2in2}Let $M_{1},M_{2}$ be $R$-modules and $N_{1},N_{2}$ be a proper
submodule of $M_{1},M_{2},$ resceptively. If $N=N_{1}\times N_{2}$ is an
$S$-weakly prime submodule of $M=M_{1}\times M_{2}$, then $N_{1}$ is an
$S_{1}$-weakly prime submodule of $M_{1}$ and $N_{2}$ is an $S_{2}$-weakly
prime submodule of $M_{2}.$
\end{proposition}

\begin{proof}
Assume that $N=N_{1}\times N_{2}$ is an $S$-weakly prime submodule of
$M=M_{1}\times M_{2}.$ Let $f_{1}\in S_{1}$ such that $0_{M_{1}}\neq
f_{1}(m_{1})\in N_{1}$ for some $m_{1}\in M_{1}.$ Take $f_{2}\in S_{2}.$ Then
$0_{M}\neq f((m_{1},0_{M_{2}}))\in N.$ Since $N$ is $S$-weakly prime, we get
either $(m_{1},0_{M_{2}})\in N$ or $f(M_{1}\times M_{2})\subseteq N.$ This
yields that $m_{1}\in N_{1}$ or $f_{1}(M_{1})\subseteq N_{1}.$ Similarly, one
can see that $N_{2}$ is $S_{2}$-weakly prime.
\end{proof}

Let $R_{i}$ be a commutative ring with identity and $M_{i}$ be an $R_{i}%
$-module for $i=1,2.$ Let $R=R_{1}\times R_{2}.$ Then $M=M_{1}\times M_{2}$ is
an $R$-module and each submodule of $M$ is in the form of $N=N_{1}\times
N_{2},$ for some submodules $N_{1}$\ of $M_{1}$\ and $N_{2}$\ of $M_{2},$ see
\cite{hojjat}.

\begin{theorem}
\label{theD1}Let $R=R_{1}\times R_{2}$ be a decomposable ring and
$M=M_{1}\times M_{2}$ be an $R$-module, where $M_{1}$ is an $R_{1}$-module and
$M_{2}$ is an $R_{2}$-module. Suppose that $N=N_{1}\times M_{2}$ is a proper
submodule of $M.$ Then the followings are equivalent:

\begin{enumerate}
\item $N_{1}$ is an $S_{1}$-prime submodule of $M_{1}.$

\item $N$ is an $S$-prime submodule of $M.$

\item $N$ is an $S$-weakly prime submodule of $M.$
\end{enumerate}
\end{theorem}

\begin{proof}
$(1)\Longrightarrow(2):$ Let $N_{1}$ be an $S_{1}$-prime submodule of $M_{1}.$

Assume that $f((m_{1},m_{2}))=(f_{1}(m_{1}),f_{2}(m_{2}))\in N=N_{1}\times
M_{2}$ for some $(m_{1},m_{2})\in M.$ Then $f_{1}(m_{1})\in N_{1}.$ As $N_{1}$
is $S_{1}$-prime, we have $m_{1}\in N_{1}$ or $f_{1}(M_{1})\subseteq N_{1}.$
Thus $(m_{1},m_{2})\in N$ or $f(M_{1}\times M_{2})\subseteq N.$

$(2)\Longrightarrow(3):$ It is clear.

$(3)\Longrightarrow(1):$\ Suppose that $N$ is an $S$-weakly prime submodule of
$M.$ Let $0_{M_{1}}\neq f_{1}(m_{1})\in N_{1}$ for some $m_{1}\in M_{1}.$ Put
$f_{2}=id_{M_{2}},$ where $id_{M_{2}}$ denotes the identity homomorphism of
$M_{2}.$ Then for every $m_{2}\in M_{2},$ we get $f((m_{1},m_{2}%
))=(f_{1}(m_{1}),f_{2}(m_{2}))\in N_{1}\times M_{2}=N.$ As $0_{M_{1}}\neq
f_{1}(m_{1}),$ we get $0_{M}\neq(f_{1}(m_{1}),f_{2}(m_{2})).$ By our
hypothesis, $(m_{1},m_{2})\in N$ or $f(M_{1}\times M_{2})\subseteq N.$
Consequently, $m_{1}\in N_{1}$ or $f_{1}(M_{1})\subseteq N_{1}.$
\end{proof}

\begin{theorem}
\label{theD2}Let $R=R_{1}\times R_{2}$ be a decomposable ring and
$M=M_{1}\times M_{2}$ be an $R$-module, where $M_{1}$ is an $R_{1}$-module and
$M_{2}$ is an $R_{2}$-module. Suppose that $N_{1},N_{2}$ be a proper submodule
of $M_{1},M_{2},$ resceptively. Then the followings are hold:

\begin{enumerate}
\item If $N=N_{1}\times N_{2}$ is an $S$-prime submodule of $M=M_{1}\times
M_{2}$, then $N_{1}$ is an $S_{1}$-prime submodule of $M_{1}$ and $N_{2}$ is
an $S_{2}$-prime submodule of $M_{2}.$

\item If $N=N_{1}\times N_{2}$ is an $S$-weakly prime submodule of
$M=M_{1}\times M_{2}$, then $N_{1}$ is an $S_{1}$-weakly prime submodule of
$M_{1}$ and $N_{2}$ is an $S_{2}$-weakly prime submodule of $M_{2}.$
\end{enumerate}
\end{theorem}

\begin{proof}
$(1)$ : It can be easily proved similar to Proposition \ref{pro1in2}.

$(2)$ : Similar to Proposition \ref{pro2in2}.
\end{proof}


\begin{thebibliography}{99}                                                                                               %


\bibitem {SF2007}S. Ebrahimi and F. Farzalipour, 2007, On weakly prime
submodules, Tamkang journal of Mathematics, 38(3), 247-252.

\bibitem {SOD2010}Shireen Ouda Dakheel, 2010, S-prime submodules and some
related concepts(Thesis).

\bibitem {ZP1988}Z. A. El-Bast and P. F. Smith, 1988, Multiplication modules,
Comm. in Algebra, 16, 755-779.

\bibitem {CPL1995}C. P. Lu, 1995, Spectra of modules, Comm. in Algebra, 23, 3741-3752.

\bibitem {JD1978}J. Dauns, 1978, Prime modules, J. reine Angew. Math., 2, 156--181.

\bibitem {GG2000}G. Gungoroglu, 2000, Strongly Prime Ideals in CS-Rings, Turk.
J. Math., 24, 233-238.

\bibitem {DE2003}D. D. Anderson and E. Smith, 2003, Weakly prime ideals,
Houston J. of Math., 29, 831--840.

\bibitem {RA2003}R. Ameri, 2003, On the prime submodules of multiplication
modules, Inter. J. of Math. and Mathematical Sciences, 27, 1715--1724.

\bibitem {PF1988}P. F. Smith, 1988, Some remarks on Multiplication modules,
Arch Math. 50, 223-235.

\bibitem {atiyah}Atiyah M. F., MacDonald I. G., 1969, Introduction to
commutative algebra, CRC Press.

\bibitem {L}Lu C. P., 1984, Prime submodules of modules, Comm. Math. Univ.
Sancti Pauli, 33: 61--69.

\bibitem {MM}McCasland R. L. and Moore M. E., 1992, Prime submodules, Comm.
Algebra, 20: 1803--1817.

\bibitem {hojjat}Mostafanasab H., Tekir U. and Oral K. H., 2016, Weakly
classical prime submodules, Kyunhpook Math. J. 56, 10085-1101.
\end{thebibliography}
\end{document}